\documentclass[11pt]{article}
\usepackage{amsmath, amssymb, amsthm, verbatim,enumerate,bbm,color,mathtools} %DIF >
\usepackage{mathrsfs}

\usepackage[backref,colorlinks,citecolor=blue,bookmarks=true]{hyperref}

\title{An Elementary Proof of a Theorem of Hardy and Ramanujan
}

\author{Asaf Cohen Antonir\thanks{School of Mathematical Sciences, Tel Aviv University, Tel Aviv, 6997801, Israel.
Email: asafc1$@$tauex.tau.ac.il} \and Asaf Shapira \thanks{School of Mathematics, Tel Aviv University, Tel Aviv 69978, Israel. Email: asafico@tau.ac.il. Supported in part
by ERC Consolidator Grant 863438 and NSF-BSF Grant 20196.}}

\date{\today}
\allowdisplaybreaks

\theoremstyle{plain}
\newtheorem{theorem}{Theorem}[section]
\newtheorem{lemma}[theorem]{Lemma}

\makeatletter
\def\moverlay{\mathpalette\mov@rlay}
\def\mov@rlay#1#2{\leavevmode\vtop{%
   \baselineskip\z@skip \lineskiplimit-\maxdimen
   \ialign{\hfil$\m@th#1##$\hfil\cr#2\crcr}}}
\newcommand{\charfusion}[3][\mathord]{
    #1{\ifx#1\mathop\vphantom{#2}\fi
        \mathpalette\mov@rlay{#2\cr#3}
      }
    \ifx#1\mathop\expandafter\displaylimits\fi}
\makeatother

\makeatletter
\renewenvironment{proof}[1][\proofname]
{\par\pushQED{\qed}
	\normalfont\topsep6\p@\@plus6\p@\relax\trivlist
	\item[\hskip\labelsep\bfseries#1\@addpunct{.}]
	\ignorespaces}
{\popQED\endtrivlist\@endpefalse}
\makeatother

\RequirePackage{color}\definecolor{RED}{rgb}{1,0,0}\definecolor{BLUE}{rgb}{0,0,1} %DIF PREAMBLE
\begin{document}
\date{}
\maketitle

\begin{abstract}
Let $Q(n)$ denote the number of integers $1 \leq q \leq n$
whose prime factorization $q= \prod^{t}_{i=1}p^{a_i}_i$ satisfies $a_1\geq a_2\geq \ldots \geq a_t$.
Hardy and Ramanujan proved that
$$
\log Q(n) \sim \frac{2\pi}{\sqrt{3}} \sqrt{\frac{\log(n)}{\log\log(n)}}\;.
$$
Before proving the above precise asymptotic formula, they studied in great detail what can be obtained concerning $Q(n)$ using
purely elementary methods, and were only able to obtain much cruder lower and upper bounds using such methods.

In this paper we show that it is in fact possible to obtain a purely elementary (and much shorter)
proof of the Hardy--Ramanujan Theorem.
Towards this goal, we first give a simple combinatorial argument,
showing that $Q(n)$ satisfies a (pseudo) recurrence relation. This enables us to replace almost all the
hard analytic part of the original proof with a short inductive argument.
\end{abstract}

\section{Introduction}

Let $\ell_k=p_1\cdot p_2 \cdots p_k$ denote the product of the first $k$ prime numbers,
and take $\mathcal{Q}$ to be the set of integers $q$ which can be expressed as $q=\ell_1^{b_1}\cdot \ell_2^{b_2}\cdots \ell_{t}^{b_t}$,
for some $t \geq 1$ and sequence of non-negative integers $b_1,\ldots,b_t$. Set $Q(n)=|\mathcal{Q}\cap [n]|$, and note that this definition of $Q(n)$ is equivalent to the one given in the abstract. The problem of bounding $Q(n)$
was introduced by Hardy and Ramanujan \cite{HR1}. As they explained, their motivation for studying this problem was its relation to highly
composite numbers \cite{Ramanujan}, its relation to variants of the partition function (see below), as well as the methods used in order to estimate $Q(n)$. The main result of \cite{HR1} was the tight asymptotic bound
\begin{equation}\label{eqHR}
\log Q(n) \sim \frac{2\pi}{\sqrt{3}} \sqrt{\frac{\log(n)}{\log\log(n)}}\;.
\end{equation}
In the first section of their paper, they studied what bounds can be obtained regarding $\log Q(n)$ using purely elementary methods.
They were only able to use such methods in order to prove the much cruder bounds
$$
C_1\sqrt{\frac{\log(n)}{\log\log(n)}} \leq \log Q(n) \leq C_2 \sqrt{\log(n)\log\log(n)}\;.
$$
They then used far more sophisticated methods (see below) in order to prove (\ref{eqHR}), not before noting that
{\em ``to obtain these requires the use of less elementary methods''}. We will show in this paper
that one can in fact prove (\ref{eqHR}) via a completely elementary and short argument. But prior
to discussing our new proof, let us first put it in perspective.

\subsection{Historical perspective}

The paper in which Hardy and Ramanjuan studied $Q(n)$ \cite{HR1} was followed a year later by the celebrated paper \cite{HarRam1918} in which they obtained their famous asymptotic formula for the partition function $p(n)$. Much as they have done in \cite{HR1}, they devoted
the first section of \cite{HarRam1918} to study what bounds can be obtained regarding $p(n)$ using purely elementary methods.
They proved that $2\sqrt{n} \leq \log p(n) \leq \sqrt{8n}$, and remarked that they can prove the right asymptotic bound
\begin{equation}\label{eqpn}
\log p(n) \sim \pi \sqrt{2n/3}
\end{equation}
using the methods they used in \cite{HR1} in their proof of (\ref{eqHR}), but that this proof {\em ``is of the difficult and delicate type''.}
Indeed, in his book on Ramanujan's work from 1940 \cite{Ram}, Hardy remarked that {\em ``It is actually true that $\log p(n) \sim \pi \sqrt{2n/3}$..., but we cannot prove this very simply''.}
This shortcoming was resolved two years later by Erd\H{o}s \cite{Erdos}, who came
up with an ingeniously simple proof of (\ref{eqpn}). His main idea was to take advantage of a certain recurrence relation satisfied
by $p(n)$ in order to bound $ p(n)$ by induction on $n$. See \cite{Nat1999,Nat2000} for further background and references.

With this perspective in mind, what we obtain in this paper can be considered an Erd\H{o}s-type proof of (\ref{eqHR}).
Unlike the case of $p(n)$, there is (to the best of our knowledge) no recurrence relation involving $Q(n)$.
This suggests that Erd\H{o}s's approach cannot be used to prove (\ref{eqHR}).
However, as we explain below, there is an approximate such relation, which turns out to be sufficient for proving (\ref{eqHR}).

\subsection{Our simplification}\label{subsec:simple}

The proof of (\ref{eqHR}) in \cite{HR1} began with an estimate for the function $\phi(\delta)=\sum_{k=1}^{\infty}\ell_k^{-\delta}$.
Hardy and Ramanujan \cite{HR1} then used their estimate for $\phi(\delta)$ in order to estimate the generating function of $Q(n)$ around zero. They then moved to the second and main step of their proof in which they proved a Tauberian theorem,
which enabled them to translate their estimate for the generating function of $Q(n)$, to the estimate for $Q(n)$ in (\ref{eqHR}).
Before describing our simplification, it is worth remarking that another indication that Hardy and Ramanujan were genuinely interested in keeping their proof as elementary as possible, is that one can easily estimate $\phi(\delta)$ while relying on the prime number theorem (see Appendix \ref{sec:phi} for the short proof). They instead found a remarkable identity (see \eqref{eq_Har_Ram}) which enabled them to estimate $\phi(\delta)$ while relying only on the elementary fact that Chebyshev's function $\vartheta(x)$ is of order $x$.

As in \cite{HR1}, our proof starts with a certain sum estimate.
Set $\Phi(\delta)=\sum_{k=1}^{\infty}\log(\ell_k)\ell_k^{-\delta}$.

\begin{lemma}\label{lem:technical}
We have $\Phi(\delta)\sim \delta^{-2}/\log(1/\delta)$ as $\delta \rightarrow 0$.
\end{lemma}
%We remark that the proof of the above lemma will relies only on the elementary fact $\log (\ell_n) = \Theta(n)$.
%The proof can be made substantially simpler and shorter if one is willing to rely on the prime number theorem which implies that $\log (\ell_n) %\sim n$.

At this point our proof departs from that of \cite{HR1} by doing away with its entire hard analytic part.
This is achieved by the following (pseudo\footnote{One expects that most of the contribution to $Q(n)$
comes from integers $q$ satisfying $\log (q) \approx \log (n)$. In this case we expect $W(n) \approx n^{Q(n)}$, which then turns
equation (\ref{eq:W}) into a ``genuine'' recurrence relation.}) recurrence relation.

\begin{lemma}\label{lem:recursion_Q}
For every integer $n$ set $W(n):=\prod_{m\in \mathcal{Q}\cap [n]}m$. Then the following relation holds
\begin{equation}\label{eq:W}
    W(n)= \prod_{k}\ell_k^{\sum _{s} Q(\lfloor n/\ell_k^{s}\rfloor)}\;.
\end{equation}
\end{lemma}

\begin{proof}
Let $Q(n,k,s)$ and $Q^{*}(n,k,s)$ be the number of integers in ${\cal Q}\cap [n]$ for which $b_k=s$ and $b_k\geq s$ respectively.
Since each integer has at most one representation as a product of $\ell_k$, we have
%\begin{align*}
\[
\pushQED{\qed}
W(n)=\prod_{s,k} \ell_k^{s\cdot Q(n,k,s)}=\prod_k\ell_k^{\sum_{s}sQ(n,k,s)}\\
    =\prod_{k}\ell_k^{\sum_{s}Q^{*}(n,k,s)}=\prod_{k}\ell_k^{\sum _{s} Q(\lfloor n/\ell_k^{s}\rfloor)}.\qedhere
\]
%\end{align*}
\end{proof}

What Lemma \ref{lem:recursion_Q} gives us is the ability to prove (\ref{eqHR})
by induction on $n$ (with the aid of Lemma \ref{lem:technical}). We prove the upper bound of (\ref{eqHR})
in the next section. Lemma \ref{lem:technical} is proved in Section \ref{sec:tech}. The proof of the lower bound of (\ref{eqHR}), which is almost identical to the proof of the upper bound, is given in Section \ref{sec:lower}.
The appendix contains proofs of a few elementary inequalities.

\section{An Upper Bound for $Q(n)$}\label{sec:upper}

Set $c=2\pi/\sqrt{3}$ and fix any $0 < \varepsilon<1/2$.
We will prove that $Q(n) \leq K e^{(1+\varepsilon)c\sqrt{{\log(Cn)}/{\log(\log(Cn))}}}$ by induction on $n$, where $K=K(\varepsilon)$ and $C=e^{10}+1$. Note that by choosing $K$ large enough, we can assume that the induction assumption holds for all $n \leq n_0(\varepsilon)$, allowing us to assume in what follows that $n \geq n_0(\varepsilon)$. We first split the contributions of $Q(n)$ as follows
\begin{align*}
    Q(n)&=Q(n^{1-\varepsilon/4})+(Q(n)-Q(n^{1-\varepsilon/4}))\\
    &\leq K e^{(1+\varepsilon)c\sqrt{{\log(Cn^{1-\varepsilon/4})}/{\log(\log(Cn^{1-\varepsilon/4}))}}}+(Q(n)-Q(n^{1-\varepsilon/4}))\\
    &\leq \frac{\varepsilon}{8} K e^{(1+\varepsilon)c\sqrt{{\log(Cn)}/{\log(\log(Cn))}}}+(Q(n)-Q(n^{1-\varepsilon/4}))\;.
\end{align*}
The first inequality holds by induction and the second holds provided $n$ is large enough.
Hence, proving that $Q(n)-Q(n^{1-\varepsilon/4}) \leq (1-\frac{\varepsilon}{8})K e^{(1+\varepsilon)c\sqrt{{\log(Cn)}/{\log(\log(Cn))}}}$
would complete the proof. Since
\[
    \log W(n) \geq (1-\varepsilon/4)\log(n)(Q(n)-Q(n^{1-\varepsilon/4}))
\]
we just need to establish that for large enough $n$
\begin{equation}\label{eq0}
    \log W(n)\leq (1-\varepsilon/2)\log(n)K e^{(1+\varepsilon)c\sqrt{{\log(Cn)}/{\log(\log(Cn))}}}\;.
\end{equation}

To simplify the presentation, we set $m=\log(Cn)$. Further, set $f(x)=\sqrt{x/\log(x)}$ and denote its derivative by $f'(x)=\frac{1-1/\log(x)}{2\sqrt{x\log(x)}}$. We will use the fact\footnote{This follows directly from Lagrange's remainder theorem, see Lemma \ref{lem:Taylor_series} in the appendix for a detailed proof.} that for every $r>10$ and $0\leq t\leq r-10$
\begin{equation}\label{eq:firstderivative}
f(r-t)\leq f(r)-t f'(r)\;.
\end{equation}

Using Lemma \ref{lem:recursion_Q} and then applying induction, we can bound $\log W(n)$ as follows
\begin{align}\label{eq:less_than_n_primes}
    \log W(n) &= \sum_{k}\sum _{s} \log(\ell_k)Q(\lfloor n/\ell_{k}^{s}\rfloor ) \leq \sum_{s}\sum _{k}\log(\ell_k) K e^{(1+\varepsilon)c f(\log(Cn/\ell_k^s))}\nonumber\\
    &\leq Ke^{(1+\varepsilon)cf(m)}\sum _{s=1}^{\infty}\sum_{k=1}^{\infty} \log(\ell_k)\ell_k^{-(1+\varepsilon)c s f'(m)}\;,
\end{align}
where the second inequality\footnote{When we write $f'(m)$, we mean substituting $x=m=\log(Cn)$ in $f'(x)=\frac{1-1/\log(x)}{2\sqrt{x\log(x)}}$.} holds by (\ref{eq:firstderivative}) with $r=m$ and $t=\log(\ell^s_k)$ noting that $m=\log(Cn)\geq 10$ and assuming that in the first line we only consider indices $k,s$ so that $\log(\ell_k^{s}) \leq \log(n) \leq m-10$ (as otherwise $Q(\lfloor n/\ell_{k}^{s}\rfloor)=0$).

Hence, to complete the proof of (\ref{eq0}) it remains to establish that the double sum in \eqref{eq:less_than_n_primes} is bounded from above by $(1-\varepsilon/2) \log(n)$. Since $c \geq 1$ and
$\alpha \lceil k\log(k+1) \rceil \leq \log(\ell_k)\leq \beta \lceil k\log(k+1)\rceil$ for all $k\geq 1$\footnote{See \cite{Nat2000} for an elementary proof that such $\alpha$ and $\beta$ exist.}, this double sum is clearly bounded from above by $S_1+S_2$ where
\[
  S_1={\sum _{s=1}^{(\log\log(n))^2}\sum_{k=1}^{\infty} \log(\ell_k)\ell_k^{-(1+\varepsilon)c s f'(m)}}\quad\text{and}\quad S_2=\sum _{s=(\log\log(n))^2}^{\infty}\sum_{k=1}^{\infty} \beta ke^{-\alpha s f'(m)k}\;.
\]
%We now show that $A \leq (1-\varepsilon)\log(n)$, and $B \leq \frac12\varepsilon \log(n)$ which would complete the proof.

\paragraph{Bounding $S_1$:}
By Lemma \ref{lem:technical} there is a $\delta_0=\delta_0(\varepsilon)$ so that $\Phi(\delta)\leq {(1+\varepsilon/12)\delta^{-2}}/{\log(1/\delta)}$ holds for every $\delta < \delta_0$.
Assume $n$ is large enough so that $(1+\varepsilon)c  (\log\log(n))^2 f'(m)\leq \delta_0$. Then,
\begin{align*}
     S_1&\leq  \sum_{s=1}^{(\log\log(n))^2} \frac{(1+\varepsilon/12)f'(m)^{-2}}{(1+\varepsilon)^2c^2 \log\left(\frac{1}{(1+\varepsilon)c s f'(m)}\right)s^2}
     \leq \frac{8(1+\varepsilon/12)^2\log(n)}{(1+\varepsilon)^2c^2}\sum_{s=1}^{\infty} \frac{1}{s^2}\leq \left(1-\varepsilon\right)\log(n)\;.
\end{align*}
The first inequality uses Lemma \ref{lem:technical} with $\delta=(1+\varepsilon)c s f'(m) ~~(\leq \delta_0)$, and the second inequality holds as $\frac{f'(m)^{-2}}{-\log\left({(1+\varepsilon)c s f'(m)}\right)}\leq 8(1+\varepsilon/12)\log(n)$ for all $1\leq s\leq (\log\log(n))^2$ and large $n$.

\paragraph{Bounding $S_2$:}
Since $S_1 \leq (1-\varepsilon)\log(n)$ it remains to prove that $S_2 \leq \frac12\varepsilon \log(n)$. Indeed
\begin{align*}
    S_2&= \sum_{s=(\log\log(n))^2}^{\infty}\frac{\beta e^{- \alpha s f'(m)}}{\left(1-e^{-\alpha s f'(m)}\right)^2}\leq  \frac{\beta}{\alpha^2 f'(m)^2}\sum _{s=(\log\log(n))^2}^{\infty} \frac{1}{s^{2}} \leq \frac{9\beta m\log(m)}{\alpha^2(\log\log(n)^2-1)}\leq \frac12\varepsilon \log(n)\;.
\end{align*}
The equality holds as $\sum_{a=1}^{\infty}at^a=\frac{t}{(1-t)^2}$, the first inequality uses the elementary
fact %\footnote{This can be proved by replacing $e^{-z}$ with its power series, see Lemma \ref{lem:estimations} in the appendix for a proof.}
(see Lemma \ref{lem:estimations}) $\frac{e^{-z}}{(1-e^{-z})^2}\leq \frac{1}{z^2}$, and the second inequality holds as for all $t$ we have $\sum_{s=t}^{\infty}\frac{1}{s^2}\leq \int_{t-1}^{\infty}\frac{1}{x^2}dx=\frac{1}{t-1}$.
%; and the last inequality holds by recalling that $m=\log(Cn)$ and assuming that $n$ is large enough.

\section{Proof of Lemma \ref{lem:technical}}\label{sec:tech}
Throughout the proof we will use Chebyshev's function, $\vartheta(x)=\sum_{p\leq x} \log(p)$, and the elementary\footnote{It is much easier to prove Lemma \ref{lem:technical} if one is willing to rely on the prime number theorem, see Appendix \ref{sec:phi}.} fact that $C_1x\leq \vartheta(x)\leq C_2x$ (see \cite{Nat2000} for a proof). We start with the following identity from \cite{HR1}:
\begin{align}\label{eq_Har_Ram}
    \sum_{k=1}^{\infty} \ell_k^{-\delta}&=\sum_{k=1}^{\infty}\frac{1-1/p^{\delta}_k}{p^{\delta}_k-1}\prod_{1 \leq i \leq k-1}p^{-\delta}_i=\frac{1}{p_1^{\delta}-1}+\sum_{k=1}^{\infty}\ell_k^{-\delta}\left(\frac{1}{p_{k+1}^{\delta}-1}-\frac{1}{p_k^{\delta}-1}\right) \nonumber\\
    &=\frac{1}{2^{\delta}-1}-\sum_{k=1}^{\infty}\ell_k^{-\delta}\int_{p_k}^{p_{k+1}}\frac{\delta x^{\delta-1}}{(x^{\delta}-1)^2}dx \nonumber\\
    &=\frac{1}{2^\delta-1}-\sum_{k=1}^{\infty}\int_{p_k}^{p_{k+1}}\frac{\delta x^{\delta-1}e^{-\delta \vartheta(p_k)}}{(x^{\delta}-1)^2}dx\;.
\end{align}
Differentiating\footnote{By the Weierstrass $M$-test, both sides of (\ref{eq_Har_Ram}), and their term by term derivatives form a (locally) uniformly convergent series on $(0,1)$. Hence, we can derive them by differentiating term by term, see \cite[Theorems 7.10 and 7.17]{Rud1976}.} both sides with respect to $\delta$ (while differentiating under the integral sign) we obtain:
\begin{align*}
    \Phi(\delta)&=-\frac{2^{\delta}\log(2)}{(2^\delta-1)^2}+\sum_{k=1}^{\infty}\int_{p_k}^{p_{k+1}}\frac{x^{\delta}e^{-\delta \vartheta(p_k)}((x^{\delta}-1)(\delta \vartheta(p_k)-1)+\delta(x^{\delta}+1)\log(x))}{(x^{\delta}-1)^3x}dx\\
    &=-\frac{2^{\delta}\log(2)}{(2^\delta-1)^2}+\sum_{k=1}^{\infty}\int_{p_k}^{p_{k+1}}\frac{x^{\delta}e^{-\delta \vartheta(x)}((x^{\delta}-1)(\delta \vartheta(x)-1)+\delta(x^{\delta}+1)\log(x))}{(x^{\delta}-1)^3x}dx\\
    &=-\frac{2^{\delta}\log(2)}{(2^\delta-1)^2}+\int_{2}^{\infty}\frac{x^{\delta}e^{-\delta \vartheta(x)}((x^{\delta}-1)(\delta \vartheta(x)-1)+\delta(x^{\delta}+1)\log(x))}{(x^{\delta}-1)^3x}dx\\
    &=-\frac{2^{\delta}\log(2)}{(2^\delta-1)^2}+\int_{2}^{\infty}\frac{\delta \vartheta(x)x^{\delta}e^{-\delta \vartheta(x)}}{(x^{\delta}-1)^2x}dx+\int_{2}^{\infty}\frac{x^\delta((x^\delta+1)\log(x^\delta)-x^\delta+1)}{(x^\delta-1)^3}\frac{e^{-\delta\vartheta(x)}}{x}dx
\end{align*}
Fixing $\varepsilon >0$ we need to prove that $|\Phi(\delta)-\frac{\delta^{-2}}{\log(1/\delta)}| \leq \frac{\varepsilon\delta^{-2}}{\log(1/\delta)}$ for all $\delta < \delta_0(\varepsilon)$. Denoting each of the above summands by $I_1,I_2,I_3$ from left to right, we now evaluate each of them.
To do so we will use the following elementary inequalities\footnote{Both inequalities can be proved by replacing $e^z$ with its power series. For a proof see Lemma \ref{lem:estimations} in the appendix.} which hold for all $z>0$:
\begin{equation}\label{eq_elem_ineq}
    \left|\frac{e^z}{(e^z-1)^2}-\frac{1}{z^2}\right|\leq \frac{1}{12}\quad\text{and}\quad\left|\frac{e^z((e^z+1)z-e^z+1)}{(e^z-1)^3}-\frac{1}{z^2}\right|\leq1\;.
\end{equation}
Applying the first inequality in \eqref{eq_elem_ineq} to $I_1$ with $z=\delta\log(2)$ we have
\begin{equation}\label{eq:first_term}
    \left|I_1+\frac{1}{\log(2)\delta^2}\right| \leq \frac{1}{12}\;.
\end{equation}
Applying the same inequality to $I_2$ with $z=\delta\log(x)$ we have
\begin{align}\label{eq:second_term}
    \left|I_2-\int_{2}^{\infty}\frac{\delta \vartheta(x)e^{-\delta \vartheta(x)}}{\delta^2\log^2(x)x}dx\right|\leq
    \int_{2}^{\infty}\frac{\delta \vartheta(x)e^{-\delta \vartheta(x)}}{12x}dx  \leq \int_{2}^{\infty}\frac{\delta C_2e^{-\delta C_1 x}}{12}dx \leq \frac{C_2}{12C_1}\;,
\end{align}
where the second inequality holds as $C_1x\leq \vartheta(x)\leq C_2x$ for all $x>0$. Finally, applying the second inequality in \eqref{eq_elem_ineq} to $I_3$ with $z=\delta\log(x)$ we have
\begin{align}\label{eq:third_term_a}
    \left|I_3-\int_{2}^{\infty}\frac{e^{-\delta \vartheta(x)}}{\delta^2\log^2(x)x}dx\right|&\leq \int_{2}^{\infty}\frac{e^{-\delta \vartheta(x)}}{x}dx\leq \frac{1}{\delta C_1}\;,
\end{align}
where the second inequality holds as $\vartheta (x)\geq C_1x$.
We conclude that provided $\delta$ is sufficiently small, then \eqref{eq:first_term},\eqref{eq:second_term}, and \eqref{eq:third_term_a} imply that
\begin{equation}\label{eq_last_integral}
    \left|\Phi(\delta)- \int_{2}^{\infty}\frac{(\delta\vartheta(x)+1)e^{-\delta \vartheta(x)}}{\delta^2\log^2(x)x}dx+\frac{1}{\log(2)\delta^2}\right|\leq\frac{\varepsilon/2}{\delta^2\log(1/\delta)}\;.
\end{equation}

For every positive $C$ let $f_C(\delta)=\int_{2}^{\infty}\frac{(\delta Cx+1)e^{-\delta Cx}}{\delta^2\log^2(x)x}dx$. Since $(x+1)e^{-x}$ is monotone decreasing for all $x>0$ and $C_1x\leq \vartheta(x)\leq C_2x$, the integral in \eqref{eq_last_integral} is bounded from above and below by $f_{C_1}(\delta),f_{C_2}(\delta)$, respectively. Hence, to conclude the proof we will prove that for every $C$, $\varepsilon>0$ and small enough $\delta$
\begin{equation}\label{eq*}
    \left|f_C(\delta)- \frac{1}{\log(2)\delta^2}-\frac{1}{\delta^2\log(1/\delta)}\right| \leq \frac{\varepsilon/2}{\delta^2\log(1/\delta)}\;.
\end{equation}
Indeed, using integration by parts and then change of variables we obtain:
\begin{align}\label{eq_whatisf}
    f_C(\delta)&=\frac{(2\delta C+1)e^{-2\delta C}}{\log(2)\delta^2}+C^2\int_{2}^{\infty}\frac{xe^{-\delta Cx}}{\log(x)}dx =\frac{(2\delta C+1)e^{-2\delta C}}{\log(2)\delta^2}+ \frac{1}{\delta^2}\int_{2\delta C}^{\infty}\frac{ze^{-z}}{\log(z/C\delta)} dz\;.
\end{align}
It remains to estimate the last integral. Since $xe^{-x}\leq 1$ and $\int_{0}^{\infty}ze^{-z}dz=1$ we have
\begin{align}\label{eq_upperboundforf}
    \int_{2\delta C}^{\infty}\frac{ze^{-z}}{\log(z/C\delta)} dz&\leq \int_{2\delta C}^{C/\log^2(1/\delta)}\frac{ze^{-z}}{\log(z/C\delta)} dz+\int_{C/\log^2(1/\delta)}^{\infty}\frac{ze^{-z}}{\log(z/C\delta)} dz\nonumber\\
    &\leq \frac{C}{\log(2)\log^2(1/\delta)}+\frac{1}{\log(1/(\delta\log^2(1/\delta)))}\leq \frac{1+\varepsilon/4}{\log(1/\delta)}\;,
\end{align}
and
\begin{align}\label{eq_lowerboundforf}
     \int_{2\delta C}^{\infty}\frac{ze^{-z}}{\log(z/C\delta)} dz\geq \int_{2\delta C}^{C\log(1/\delta)}\frac{ze^{-z}}{\log(z/C\delta)} dz\geq \int_{2\delta C}^{C\log(1/\delta)}\frac{{ze^{-z}}}{\log(\log(1/\delta)/\delta)} dz \geq \frac{1-\varepsilon/4}{\log(1/\delta)}\;,
\end{align}
for small enough $\delta$. Combining \eqref{eq_whatisf},\eqref{eq_upperboundforf},\eqref{eq_lowerboundforf} we obtain \eqref{eq*} for every small enough $\delta$.

\section{A Lower Bound for $Q(n)$}\label{sec:lower}

The proof is almost identical to the proof in Section \ref{sec:upper}. The only difference is that instead of (\ref{eq:firstderivative}) we now use (\ref{eq:secondderivative}), hence we need to account for $f''$. Set $c=2\pi/\sqrt{3}$ and fix any $0 < \varepsilon<1/2$.
We will prove that $Q(n) \geq \frac{1}{K} e^{(1-\varepsilon)c\sqrt{\log(Cn)/\log\log(Cn)}}$ by induction on $n$, where $K=K(\varepsilon)$, and $C=e^{10^3}+1$. Note that by choosing $K$ large enough, we can assume that the induction assumption holds for all $n \leq n_0(\varepsilon)$, allowing us to assume in what follows that $n \geq n_0(\varepsilon)$.

To simplify the presentation we set $m= \log(Cn)$. Further, set $f(x)=\sqrt{x/\log(x)}$, denote its derivative by $f'(x)=\frac{1-1/\log(x)}{2\sqrt{x\log(x)}}$, and its second derivative by $f''(x)=\frac{-1+3/\log^2(x)}{4x\sqrt{x\log(x)}}$.
We will next use the fact\footnote{This follows by direct computation, see Lemma \ref{lem:Taylor_series} in the appendix for a detailed proof.} that for $r\geq 10^3$ and any $0\leq t\leq r/10$
\begin{equation}\label{eq:secondderivative}
f(r-t)\geq f(r)-t f'(r)+t^2f''(r)\;.
\end{equation}

Setting $s'=\log\log(n)$ and $k'=\log^{3/4}(n)$ we have $\ell_{k'}^{s'}\leq n^{1/10}$ for large enough $n$, and hence
\begin{align}\label{eq:less_than_n_primes_2}
    \log(n)Q(n)&\geq \log\left(W(n)\right) = \sum_{k}\sum _{s} \log(\ell_k)Q(\lfloor n/\ell_{k}^{s}\rfloor)\nonumber\\
    &\geq \sum_{s\leq s'}\sum _{k\leq k'}\log(\ell_k) \frac{1}{K} e^{(1-\varepsilon)cf(\log(C n)-(1+\varepsilon/4)\log(\ell_k^s))}\nonumber\\
    &\geq \frac{1}{K}e^{(1-\varepsilon)cf(m)}\sum _{s\leq s'}\sum_{k\leq k'} \log(\ell_k)e^{-(1-\varepsilon/2)c \log(\ell_k^s) f'(m)} \cdot e^{2c \log^2(\ell_k^s)f''(m)}\nonumber\\
    &\geq \frac{1}{K}e^{(1-\varepsilon)cf(m)}\sum _{s\leq s'}\sum_{k\leq k'} \log(\ell_k)\ell_k^{-(1-\varepsilon/2)c s f'(m)} \left(1+2c \log^2(\ell_k^s)f''(m)\right)\;,
\end{align}
where the equality holds by Lemma \ref{lem:recursion_Q}, the second inequality holds by induction, the fact that for all $x>2$ we have $\log(\lfloor x\rfloor )\geq \log(x)-2/x$, and by assuming $n$ is large enough so that $\ell_{k}^{s}/n\leq \frac{1}{8}\varepsilon \log(\ell_{k}^{s})$ for all $s\leq s'$ and $k\leq k'$; the third inequality holds\footnote{As in Section \ref{sec:upper}, $f'(m)$ and $f''(m)$ mean plugging $x=m=\log(Cn)$ into the first/second derivatives of $f$.}
by (\ref{eq:secondderivative}) with
$r=m$, $t=(1+\varepsilon/4)\log(\ell_k^s)$ noting that $m=\log(Cn)\geq 10^3$ and $\log(\ell_{k'}^{s'}) \leq \log(n^{1/10}) \leq m/10$, and the last inequality holds as %$f''(m)\leq 0$ provided $n$ is large enough, and as
for all $x$ we have $1+x\leq e^{x}$.

Hence, to complete the proof it remains to establish that the double sum in \eqref{eq:less_than_n_primes_2} is bounded from below by $\log(n)$.
Since $(1-\varepsilon/2)c>1$ and since $\alpha \lceil k\log(k+1)\rceil \leq \log(\ell_k)\leq \beta \lceil k\log(k+1)\rceil$ for all $k\geq 1$, the double sum in \eqref{eq:less_than_n_primes_2} is at least $S_1-S_2+S_3$ where (note that $f''(m)<0$)
\[
    S_1= \sum _{s=1}^{s'}\sum_{k=1}^{\infty} \log(\ell_k)\ell_k^{-(1-\varepsilon/2)c s f'(m)} \quad\text{,}\quad S_2=\sum_{s=1}^{s'}\sum_{k=k'}^{\infty} \beta ke^{-\alpha s f'(m)k}\;,
\]
\[
    S_3=\sum _{s=1}^{s'}\sum_{k=1}^{\infty}  2\beta^3 c s^2k^3f''(m)e^{-\alpha  s f'(m) k}\;.
\]
\paragraph{Bounding $S_1$:}
By Lemma \ref{lem:technical} there is a $\delta_0=\delta_0(\varepsilon)$ so that $\Phi(\delta)\geq  {(1-\varepsilon/4)\delta^{-2}}/{\log(1/\delta)}$ holds for every $\delta < \delta_0$.
Assume $n$ is large enough so that $(1-\varepsilon/2)c  \log\log(n) f'(m)\leq \delta_0$. Then,
\begin{align*}
     S_1&\geq  \sum_{s=1}^{\log\log(n)} \frac{(1-\frac{\varepsilon}{4})f'(m)^{-2}}{(1-\frac{\varepsilon}{2})^2c^2 \log\left(\frac{1}{(1-\frac{\varepsilon}{2})c s f'(m)}\right)s^2} \geq \frac{8(1-\frac{\varepsilon}{4})^{3/2}\log(n)}{(1-\frac{\varepsilon}{2})^2c^2}\sum_{s=1}^{\log\log(n)} \frac{1}{s^2} \geq \left(1+\frac{\varepsilon}{2}\right)\log(n)\;,
\end{align*}
where the first inequality holds by Lemma \ref{lem:technical} applied with $\delta=(1-\varepsilon/2)c s f'(m) ~~(\leq \delta_0)$, the second inequality holds as $\frac{f'(m)^{-2}}{-\log\left({(1-\varepsilon/2)c s f'(m)}\right)}\geq 8\sqrt{1-\varepsilon/4}\log(n)$ for all $1\leq s\leq \log\log(n)$ and large $n$. The last inequality holds provided $n$ is large enough so that $\sum_{s=1}^{\log\log(n)}\frac{1}{s^2}\geq \sqrt{1-\varepsilon/4}\cdot \pi^2/6$.

\paragraph{Bounding $S_2$:}
Observe that the following holds for large enough $n$:
\begin{align*}
    S_2&\leq \beta \log\log(n)\cdot \frac{e^{-\alpha f'(m)\log^{3/4}(n)}}{\left(1-e^{-\alpha f'(m)}\right)^2}\leq  \frac{\beta \log\log(n)\cdot e^{-\alpha f'(m)\log^{3/4}(n)}}{\alpha ^2\left(f'(m)-f'(m)^2\right)^2} \leq \frac{1}{4}\varepsilon \log(n)\;,
\end{align*}
where the first inequality holds as for all $M>0$ and $0<x<1$ we have $\sum_{k=M}^{\infty} k x^{k}\leq \frac{x^M}{(1-x)^2}$, the second inequality holds as for all $0<x<1$ we have $(1-e^{-x})^2>(x-x^2)^2$.

\paragraph{Bounding $S_3$:}
Since $S_1-S_2\geq (1+\varepsilon/4)\log(n)$, it is enough to prove that $S_3\geq -\frac{1}{4}\varepsilon \log(n)$. Indeed,
\begin{align*}
    S_3&\geq 2\beta^3 cf''(m) \sum _{s=1}^{\log\log(n)}s^2\frac{6e^{-\alpha s f'(m)}}{\left(1-e^{-\alpha s f'(m)}\right)^4} \geq\frac{36\beta^3 cf''(m)}{\alpha^4f'(m)^4}  \sum_{s=1}^{\infty} \frac{1}{s^2}\geq -\frac{1}{4}\varepsilon \log(n)\;,
\end{align*}
where the first and second inequalities hold as for all $0<z< 1$ we have\footnote{This follows by taking the derivative of the identity $\sum_{k=0}^{\infty}z^{k}=\frac{1}{1-z}$ three times with respect to $z$.} $\sum_{k=0}^{\infty}k^3z^{k}\leq \frac{6z}{(1-z)^4}$ and\footnote{See Lemma \ref{lem:estimations} in the appendix for a short proof.} $e^{-z}/(1-e^{-z})^4\leq 3/z^4$. The last inequality holds for large enough $n$.

\appendix

\section{Some Elementary inequalities}

\begin{lemma}\label{lem:Taylor_series}
    Let $f(x)=\sqrt{x/\log(x)}$. Then, for every $m\geq 10$ and $0\leq x\leq m-10$ we have:
    \[
        f(m-x)\leq f(m)-xf'(m)\;.
    \]
    For all $m\geq 10^3$ and $0\leq x\leq m/10$ we have
    \[
        f(m-x)\geq f(m)-xf'(m)+x^2f''(m)\;.
    \]
\end{lemma}

\begin{proof}
By Lagrange's remainder theorem, for every $m>1$ and $x$ such that $0\leq x\leq m$ there is $m-x\leq c\leq m$ such that
\begin{align*}
    f(m-x)&=f(m)- xf'(m)+x^2f''(c)/2\;.
\end{align*}
Noting that $f''(z)=\frac{3-\log^2(z)}{4z^{3/2}\log^{5/2}(z)}$ is negative for all $z>e^{\sqrt{3}}$ we obtain the first part of the lemma for $m>e^{\sqrt{3}}$ and $x>0$ with $m-x>e^{\sqrt{3}}$.

For the second part of the lemma, for every $m$ and $x\leq m$ we let
\[
    g_m(x)\coloneqq f(m-x)-f(m)+xf(m)-x^2f''(m)\;.
\]
Note that $g_m(0)=g_m'(0)=0$. Hence, proving that for all $m\geq 10^3$ and $0\leq x\leq m/10$ we have $g''_m(x)>0$ implies the second part of the lemma.
Indeed, let $m\geq 10^3$ and let $x=\varepsilon m $ with $0\leq \varepsilon\leq 1/10$. We have
\begin{align*}
    g''_m(\varepsilon m)&=\frac{3-\log^2((1-\varepsilon)m)}{4(1-\varepsilon)^{3/2}m^{3/2}\log^{5/2}((1-\varepsilon)m)}-\frac{3-\log^2(m)}{2m^{3/2}\log^{5/2}(m)}\\
    &\geq \frac{1}{4m^{3/2}}\left(\frac{1}{(1-\varepsilon)^{3/2}}\cdot \frac{3-\log^{2}(9m/10)}{\log^{5/2}(9m/10)}- 2\cdot \frac{3-\log^{2}(m)}{\log^{5/2}(m)}\right)\\
    &\geq \frac{1}{4m^{3/2}}\cdot {\frac{2(1-1/10)-(1-\varepsilon)^{-3/2}(1+1/10)}{\sqrt{\log(m)}}}>0
\end{align*}
where the first inequality holds as $\frac{3-\log^2(x)}{\log^{5/2}(x)}$ is monotone increasing for $x>e^{\sqrt{15}}$ and as $(1-\varepsilon)m\geq 900\geq e^{\sqrt{15}}$, and the second inequality holds as $m>10^3$ which implies both $ \frac{3-\log^2(9m/10)}{\log^{5/2}(9m/10)}\geq \frac{-(1+1/10)}{\sqrt{\log(m)}}$ and $\frac{3-\log^2{m}}{\log^{5/2}(m)}\leq \frac{-(1-1/10)}{\sqrt{\log(m)}}$, and the last inequality holds as $\varepsilon\leq 1/10$.
\end{proof}

\begin{lemma}\label{lem:estimations}
For all $z>0$ we have
\begin{equation}\label{eq1}
    \frac{1}{z^2}-\frac{1}{12}\leq \frac{e^z}{(e^z-1)^2}=\frac{e^{-z}}{(e^{-z}-1)^2}\leq \frac{1}{z^2}\;,
\end{equation}
\begin{equation}\label{eq2}
    \frac{1}{z^2}-\frac{1}{12}\leq \frac{e^z((e^z+1)z-e^z+1)}{(e^z-1)^3}\leq \frac{1}{z^2}+1\;,
\end{equation}
and for all $0<z<1$ we have
\begin{equation}\label{eq3}
        \frac{e^{-z}}{(1-e^{-z})^4}=\frac{e^{3z}}{(e^z-1)^4}\leq \frac{3}{z^4}\;.
    \end{equation}
\end{lemma}
\begin{proof}
The proof of all five inequalities are similar; we replace $e^z$ with its power series, and then compare the coefficients on both
sides of the inequality. We prove the upper bound/lower bound of \eqref{eq1} with respect to $\frac{e^{z}}{(e^{z}-1)^2}$. The upper bound of \eqref{eq1} is equivalent to
$
    z^2e^z\leq e^{2z}-2e^{z}+1.
$
By expanding each side to its power series, we need to show that
\[
    \sum_{n=2}^{\infty}\frac{z^n}{(n-2)!} \leq \sum_{n=2}^{\infty}\frac{2^n-2}{n!}z^n\;.
\]
To see this, note that the coefficient of $z^n$ on the left-hand side is always at most the coefficient of $z^n$ on the right-hand side for all $n\geq 2$ as clearly $n(n-1)\leq (2^n-2)$ for all $n\geq 2$.

The lower bound of \eqref{eq1} is equivalent to
$
    12e^{2z}-24e^{z}+12\leq z^2(e^{2z}+10e^{z}+1).
$
By expanding each side to its power series, we wish to show that
\[
    \sum_{n= 3}^{\infty}\frac{12\cdot(2^{n}-2)}{n!}z^n\leq \sum_{n=3}^{\infty}\frac{n(n-1)(2^n+10)}{n!}z^n\;.
\]
To see this, we show that each of the coefficients of $z^n$ on the left-hand side is at most the ones of $z^n$ on the right-hand side.
The coefficient of $z^3$ on the left-hand side is $72$ and on the right-hand side it is $108$; for $n\geq 4$ we have $12\leq n(n-1)$ and clearly for all $n$ we have $2^n-2\leq 2^n+10$, implying that the coefficient of $z^{n}$ on the left-hand side is at most the coefficient of $z^{n}$ on the right-hand side.

Observe that \eqref{eq2} follows directly from inequality \eqref{eq1} and the following inequality
\begin{equation}\label{eq4}
    1\leq \frac{(e^z+1)z-e^z+1}{e^z-1}\leq 1+z^2\;.
\end{equation}
The upper bound in \eqref{eq4} is equivalent to
$
    2+z+z^2\leq (2-z+z^2)e^z.
$
By expanding each side to its power series, it is enough to show that $2n(n-1)-n+1\geq 0$ for all $n\geq 2$. This clearly holds. As for the lower bound, note that as $z>0$ the inequality is equivalent to
$
    2e^z-2\leq (e^z+1)z.
$
By expanding each side to its power series, we wish to show that
\[
    \sum_{n=2}^{\infty}\frac{2z^n}{n!}\leq \sum_{n=2}^{\infty}\frac{z^n}{(n-1)!}\;,
\]
This clearly holds as for all $n\geq 2$ the coefficient of $z^n$ on the left-hand side is at most the coefficient of $z^n$ on the right-hand side.

Clearly, to prove \eqref{eq3} it is enough to prove that $\frac{e^{2z}}{(e^{z}-1)^2}\leq \frac{1}{z^2}+\frac{1}{z}+1$, as then using inequality \eqref{eq1}, we are done.
The above is equivalent to proving that
    $
        2z^2e^{z}\leq (1+z)(e^{2z}-2e^{z}+1)+z^2.
    $
    By expanding each side to its power series, we wish to show that
    \[
        \sum_{n=4}^{\infty}\frac{2n(n-1)}{n!}z^n\leq \sum_{n=4}^{\infty} \frac{2^n-2+n(2^{n-1}-2)}{n!}z^{n}\;.
    \]
    As in previous cases, this holds since the coefficient of $z^{n}$ on the left-hand side is at most as large as the coefficient of $z^{n}$ on the right-hand side.
\end{proof}

\section{Estimating $\phi(\delta)$ and $\Phi(\delta)$ Assuming the Prime Number Theorem}\label{sec:phi}

Hardy and Ramanujan \cite{HR1} proved that
$\phi(\delta)\sim \delta^{-1}/\log(1/\delta)$, while only relying on elementary facts.
Let us show that if one is willing to rely
on the prime number theorem, which is equivalent to the statement that $\log(\ell_k)\sim k\log k$, then
one can prove this estimate quite easily. Using the same approach, one can also easily prove Lemma \ref{lem:technical} (again, assuming
the prime number theorem).

We claim that the following holds
$$
\delta^{-1}/\log(1/\delta) \sim \int_{1}^{\infty} e^{-\delta x\log x} dx \sim \sum_{k \geq 2}e^{-\delta k\log k} \sim \phi(\delta).
$$
The first estimate is proved below, the second estimate follows from the fact that $e^{-\delta x}$ is a monotone function,
and the last estimate follows from the fact that $\log(\ell_k)\sim k \log k$, that is, from the fact that $\phi(\delta)$
is bounded from above (resp.\ below) by $\sum_{k}e^{-\delta(1-o(1)) k\log k}$ (resp.\ $\sum_{k}e^{-\delta(1+o(1)) k\log k}$).

It thus remains to estimate the above integral. We prove that for every $\varepsilon$ and all small enough $\delta < \delta_0(\varepsilon)$,
this integral is bounded from above by $(1+\varepsilon)\delta^{-1}/\log(1/\delta)$. The proof of the lower bound is identical.
Let $a_0=a_0(\varepsilon)$ be such that $\log(x)\geq (1-\frac14\varepsilon)\log(x\log(x))$ for every $x \geq a_0$.
Then
\begin{align*}
  \int_{1}^{\infty} e^{-\delta x\log x} dx &\sim \int_{a_0}^{\infty} e^{-\delta x\log x} dx \leq \frac{1}{(1-\varepsilon/8)\delta} \int_{\delta a_0\log(a_0)}^{\infty} \frac{e^{-z}}{\log(z/\delta)} dz \\
   &\leq \frac{1+\varepsilon/4}{\delta} \left(\int_{\delta a_0\log(a_0)}^{1/\log^2(1/\delta)}\frac{e^{-\delta z}}{\log(z/\delta)}dz+\int_{1/\log^2(1/\delta)}^{\infty} \frac{e^{-z}}{\log(z/\delta)} dz\right)\\
   &\leq \frac{1+\varepsilon/4}{\delta}\left(\frac{C}{\log^2(1/\delta)}+\frac{1}{\log(1/\delta)-2\log\log(1/\delta)}\int_{0}^{\infty}e^{-z}dz\right)\\
   &\leq \frac{\varepsilon/2}{\delta\log(1/\delta)}+\frac{1+\varepsilon/2}{\delta\log(1/\delta)}\int_{0}^{\infty}e^{-z}dz= \frac{1+\varepsilon}{\delta\log(1/\delta)}\;,
\end{align*}
where in the first inequality we used the substitution $z=\delta x \log x$, which, by the choice of $a_0$, guarantees that
$dz=\delta(\log(x)+1)dx\geq \delta(1-\varepsilon/8)\log(z/\delta)dx$, and in the third line $C$ is a constant that depends only on $a_0$.
\end{document}